\documentclass[11pt,oneside,reqno]{amsart}
 \usepackage{amstext,amsmath,amssymb,amsthm} %esint
\usepackage{latexsym}
\usepackage{exscale}
\usepackage{color}

\usepackage[margin=1in]{geometry} 
\usepackage{amsmath,amsthm,amssymb, graphicx, multicol, array}
 
\usepackage{pgfplots}
\pgfplotsset{compat=1.6}

\pgfplotsset{soldot/.style={color=blue,only marks,mark=*}} \pgfplotsset{holdot/.style={color=blue,fill=white,only marks,mark=*}}

\RequirePackage[cp1251]{inputenc}
\RequirePackage{srcltx}

\numberwithin{equation}{section}
\RequirePackage{srcltx}

\RequirePackage[cp1251]{inputenc}

\numberwithin{equation}{section}

\newcommand\Ann{\textnormal{Ann}}

\newcommand\E{\mathbb{E}}

\newcommand\Z{\mathbb{Z}}
\newcommand\N{\mathbb{N}}

\renewcommand\P{\mathbb{P}}

\usepackage{mathtools}

\DeclarePairedDelimiter\floor{\lfloor}{\rfloor}

\theoremstyle{plain}
\newtheorem{Th}{Theorem}[section]
\newtheorem{Lemma}[Th]{Lemma}

\newtheorem{Prop}[Th]{Proposition}

 \theoremstyle{definition}
\newtheorem{Def}[Th]{Definition}

\newtheorem{?}[Th]{Problem}

\renewcommand\hat{\widehat}
\def\XXint#1#2#3{{\setbox0=\hbox{$#1{#2#3}{\int}$ }
\vcenter{\hbox{$#2#3$ }}\kern-.6\wd0}}

\newcounter{cnstcnt}

\newcommand\txtarrowlr[2]{\xleftrightarrow[ \text{ #2 }]{ \text{ #1 } } }

\usepackage[foot]{amsaddr} %Author title and affil

\author{Michael Damron\(^1\)}
\address{\(^1\)Georgia Institute of Technology, 686 Cherry St., Atlanta, GA 30332 USA}
\email[\(^1\)]{mdamron6@gatech.edu}

\author{David Harper\(^2\)}
\address{\(^2\)Georgia Institute of Technology, 686 Cherry St., Atlanta, GA 30332 USA}
\email[\(^2\)]{dharper40@gatech.edu}

\begin{document}

	\title{Non-Optimality of Invaded Geodesics in 2d Critical First-Passage Percolation}
	%\author{Michael Damron, David Harper}
	\date{}
	% \email{dharper40@gatech.edu}
	%
	%
%	\subjclass[2017]{}

	\begin{abstract}
	 We study the critical case of first-passage percolation in two dimensions. Letting $(t_e)$ be i.i.d.~nonnegative weights assigned to the edges of $\mathbb{Z}^2$ with $\mathbb{P}(t_e=0)=1/2$, consider the induced pseudometric (passage time) $T(x,y)$ for vertices $x,y$. It was shown in \cite{Asymptotics} that the growth of the sequence $\mathbb{E}T(0,\partial B(n))$ (where $B(n) = [-n,n]^2$) has the same order (up to a constant factor) as the sequence $\mathbb{E}T^{\text{inv}}(0,\partial B(n))$. This second passage time is the minimal total weight of any path from 0 to $\partial B(n)$ that resides in a certain embedded invasion percolation cluster. In this paper, we show that this constant factor cannot be taken to be 1. That is, there exists $c>0$ such that for all $n$,
	 \[
	 \mathbb{E}T^{\text{inv}}(0,\partial B(n))  \geq (1+c) \mathbb{E}T(0,\partial B(n)).
	 \]
	This result implies that the time constant for the model is different than that for the related invasion model, and that geodesics in the two models have different structure.
%	 
%	 
%	 
%	 
%	 In \cite{Asymptotics} it was shown that the passage time in critical first-passage percolation has the same first-order behavior as the passage time of an optimal path constrained to lie in an embedded invasion cluster. However the question remained as to whether the result could be strengthened to an asymptotic result. We show here a negative answer to this possible. 
	\end{abstract}	 
	
	\maketitle
	
%%%%%%%%%%%%%%%%%%%%%%%%%%%%%%%%%%%%%%%%%%%%%%%%%%%%%%%%%%%%%%%%%%%%%%%%%%%%%%%%%%%%%%%%%%%%%%%%%%	BEGIN (1)
\section{Introduction}

\subsection{Background and main result}

%\textbf{Here would go all the necessary definitions to formulate the main result.}

%{\color{green} Things to do:
%\begin{enumerate}
%\item Can remove $\pi$ from figure 2. COMPLETED
%\item Remove unnecessary labels in definition of $E_k$. POSSIBLY COMPLETED? 
%\item Add Georgia Tech affiliation, email addresses. COMPLETED
%\end{enumerate}
%}

We begin with the definition of first-passage percolation (FPP). Consider the integer lattice $\mathbb{Z}^2$ with $\mathcal{E}^2$ denoting the set of nearest-neighbor edges, and let $(t_e)_{e \in \mathcal{E}^2}$ be an i.i.d.~family of nonnegative random variables (edge-weights) with common distribution function $F$. For $x,y \in \Z^2$, a (vertex self-avoiding) path from $x$ to $y$ is a sequence $ ( v_0 , e_1 , v_1,\dots ,e_n,v_n)$, where the $v_i$'s, $i = 1,...,n-1$, are distinct vertices in $\Z^2$ which are different from $x$ or $y$, and $v_0 = x $, $v_n = y$, $e_i = \{v_{i-1}, v_i\} \in \mathcal{E}^2$. If $x = y$, the path is called a (vertex self-avoiding) circuit. We define the passage time of a path $\gamma$ to be $ T(\gamma) = \sum_{i =1}^n t_{e_i} $. For any $A,B \subset \Z^2$ we denote
\[ 
T(A,B) = \inf \{ T(\gamma) : \gamma \text{ is a path from a vertex in } A \text{ to a vertex in } B \}. 
\] 
For $A = \{x \} $ we write $T(x,B) $ to mean $T(\{x\},B)$ and similarly for $T(A,x)$. A geodesic from $A$ to $B$ is a path $\gamma$ from $A$ to $B$ such that $T(\gamma) = T(A,B)$. Note that $T=T(x,y)$ as a function of vertices $x,y$ is a psuedometric, and is a.s.~a metric if and only if $F(0)=0$. Thus $(\mathbb{Z}^2,T)$ can be regarded as a random pseudometric space.

FPP is studied as a model for fluid flow in a porous medium, or of the spread of a stochastic growth, such as a bacterial infection. It was introduced in 1965 by Hammersley and Welsh \cite{HW} and since then, researchers have developed some of the basics of the theory including asymptotics for $T(0,x)$ as $x \to \infty$, shape theorems, fluctuations of $T$, and geometry of geodesics (see \cite{ADH} for a recent survey). Analysis of the model is quite different depending on the relationship between $F(0)$ and the critical value $p_c=1/2$ for two-dimensional Bernoulli percolation. In the supercritical case, where $F(0)>1/2$, there exists an infinite cluster (component) of edges with zero weight, and one can then show that $T(0,x)$ is stochastically bounded in $x$. (To reach $x$ from 0, just enter the infinite cluster and travel near to $x$ in zero time.) In the subcritical (and most studied) case, where $F(0)<1/2$, $T(0,x)$ grows linearly in $x$, and there are many results and conjectures about the precise rate of growth. 

The critical case which we study here, where $F(0)=p_c = 1/2$, is considerably more subtle and it is closely related to near-critical and critical bond percolation. There is no infinite cluster of zero-weight edges, but there are large zero-weight clusters on all scales. Here, the usual ``time constant,'' defined as
\[
\mu = \lim_{n \to \infty} \frac{T(0,ne_1)}{n}
\]
is known to be zero (from Kesten's result \cite[Theorem~6.1]{aspects} that $\mu = 0$ if and only if $F(0) \geq 1/2$), so it is natural to ask for the correct (sublinear) growth rate of $T$. Instead of $T(0,ne_1)$, it is more convenient to consider $T(0,\partial B(n))$, where $B(n) = [-n,n]^2$, and after important work of Chayes-Chayes-Durrett \cite{CCD} and Zhang \cite{Zhangdouble}, it was shown by Damron-Lam-Wang in \cite[Theorem~1.2]{Asymptotics} that
\begin{equation}\label{eq: asymptotics}
\mathbb{E}T(0,\partial B(n)) \asymp \sum_{k=1}^{\lfloor \log n \rfloor} F^{-1}\left( \frac{1}{2} + \frac{1}{2^k}\right),
\end{equation}
where $a_n \asymp b_n$ means that $b_n/a_n$ is bounded away from 0 and $\infty$, and $F^{-1}$ is the following generalized inverse of $F$:
\[
F^{-1}(t) = \inf \{ x : F(x) \geq t \} \text{ for }t > 0.
\]
To prove this result, the authors introduced an embedded invasion percolation cluster (an infinite connected subgraph $I$ of $\mathbb{Z}^2$ containing the origin which we will define in the next section), and showed that 
%{\color{green} Check whether this is actually shown, or whether invasion has $n/2$ or something.}
\begin{equation}\label{eq: invasion_comparison}
\mathbb{E}T(0,\partial B(n)) \asymp \mathbb{E}T^{\text{inv}}(0,\partial B(n)),
\end{equation}
where $T^{\text{inv}}$ is defined analogously to $T$, but only using paths which remain in $I$ (see \eqref{eq: T_inv_def}). They then argued that $\mathbb{E}T^{\text{inv}}(0,\partial B(n)) \asymp$ the right side of \eqref{eq: asymptotics}.

The main result of our work implies that the symbol $\asymp$ in the comparison \eqref{eq: invasion_comparison} cannot be replaced by the stronger $\sim$. In other words, the ratio of the left and right sides does not converge to 1: the invasion passage time is only a good approximation for the true passage time up to a constant factor. Therefore, local properties of geodesics or the passage time cannot be studied by a comparison to invasion.

%% Below was the previous version of main result. It was modfied to remove the  need for a moment assumption.
%\begin{Th}\label{thm: main_thm}
%Suppose that $F(0)=p_c$ and that $\mathbb{E}\min\{t_1, \dots, t_4\}^\alpha<\infty$ for some $\alpha>1$, where the $t_i$'s are i.i.d. copies of $t_e$. Then there exist $c_1,C_1>0$ such that for all $n \geq 0$,
%\[
%c_1 \mathbb{E}T(0,\partial B(n)) \leq \mathbb{E}[T^{\text{inv}}(0,\partial B(n)) - T(0, \partial B(n))] \leq C_1 \mathbb{E}T(0,\partial B(n)).
%\]
%\end{Th}

\begin{Th}\label{thm: main_thm}
Suppose that $F(0)=p_c = 1/2$. There exists $c_{\ref{thm: main_thm} . 1}>0$ such that for all large $n$,
\[
\mathbb{E}[T^{\text{inv}}(0,\partial B(n)) - T(0, \partial B(n))]  \geq c_{\ref{thm: main_thm} . 1}  \sum_{k=1}^{\lfloor \log n \rfloor} F^{-1}\left( \frac{1}{2} + \frac{1}{2^k}\right).
\]
\end{Th}

In Section~\ref{sec: coupling} below, we define the embedded invasion percolation model, and give some important properties of critical and near-critical percolation used in the paper. In Section~\ref{sec: outline}, we give an outline of the proof of Theorem~\ref{thm: main_thm}, and in Section~\ref{sec: main_proof} we give the full proof. Throughout the paper, constants will be denoted by $c$ or $C$ depending on whether they are large or small, and their subscripts refer to the sections in which they are defined.

%% Outline (Start)

%% Outline (End)

\subsection{Coupled percolation models}\label{sec: coupling}

We will couple the FPP model on $(\Z^2, \mathcal{E}^2)$ with invasion percolation and Bernoulli percolation.  To describe the coupling, we consider the probability space $ (\Omega , \mathcal{F} , \P )$ , where $\Omega = [0,1]^{\mathcal{E}^2} $, $\mathcal{F} $ is the product Borel sigma-field, and $\P = \prod_{e \in \mathcal{E}^2 } \mu_e $, where each $\mu_e $ is the uniform measure on $[0,1]$. Write $\omega = (\omega_e)_{e \in \mathcal{E}^2 } \in \Omega $ so that the coordinates $(\omega_e)$ are i.i.d.~uniform $[0,1]$ random variables, and define the edge weights as $t_e = F^{-1}(\omega_e) $ for $e \in \mathcal{E}^2$, so that the collection $(t_e)$ is i.i.d.~with common distribution function $F$.

The uniform variables $(\omega_e)$ will be used for two models: invasion percolation and Bernoulli percolation.
\begin{itemize}
\item {\bf Invasion percolation} is a another model for a stochastic growth which, unlike FPP, follows a greedy algorithm. Because of its relation to critical Bernoulli percolation, it is known as a model of self-organized criticality. 

To define the growth, we first define, the edge boundary $\Delta G$ of an arbitrary subgraph $G = (V,E)$ of $(\Z^2,\mathcal{E}^2)$ by
\[ 
\Delta G = \{ e \in \mathcal{E}^2 : e \notin E , e \text{ has an endpoint in } V \}. 
\]
Next, the invasion proceeds in discrete time, as a sequence $(G_n)_{n = 0}^{\infty}$ of subgraphs of $(\mathbb{Z}^2,\mathcal{E}^2)$ as follows. Let $G_0 = ( \{0\}, \emptyset ) $. Given $G_i = (V_i,E_i)$, we define $E_{i+1} = E_i \cup \{e_{i+1}\} $, where $e_{i+1}$ is the a.s.~unique edge with $\omega_{e_{i+1}} = \min \{ \omega_{e} : e \in \Delta G_i \}$, and let $G_{i+1}$ be the graph induced by $E_{i+1}$. The graph $I = \cup_{i=0}^{\infty} G_i $ is called the invasion percolation cluster (at time infinity). 

If $A,B$ are subsets of $\mathbb{Z}^2$, we set
\begin{equation}\label{eq: T_inv_def}
T^{\text{inv}}(A,B) = \inf_{\gamma: A\txtarrowlr{\textit{I}}{} B} T(\gamma),
\end{equation}
where the infimum is over all paths from $A$ to $B$ which remain in the invasion $I$. (Here, $\inf \emptyset$ is defined as $+\infty$.) This $T^{\text{inv}}$ is the passage time to which we compare $T$ in Theorem~\ref{thm: main_thm}. 

\item {\bf Bernoulli percolation} is a simple model for a random network. The usual setup for Bernoulli percolation requires us to choose a parameter $p \in [0,1]$ and then independently declare each edge in our graph $(\mathbb{Z}^2,\mathcal{E}^2)$ to be open with probability $p$ and closed with probability $1-p$. Using our uniform variables, we can couple all of these models (for different values of $p$) to the other models described above. For each $e \in \mathcal{E}^2$ and $p \in [0,1]$, we say that an edge $e$ is $p$-open in $\omega$ if $\omega_e \leq p$ and otherwise say that $e$ is $p$-closed. Then for any fixed $p$, the collection of $p$-open edges has the same distribution as the set of open edges in Bernoulli percolation with parameter $p$.
\end{itemize}

Next we give a couple of definitions that will help us work with these models. For $p \in [0,1]$, a path (or circuit) is said to be $p$-open (respectively $p$-closed) if all its edges are $p$-open (respectively $p$-closed). Recall that all paths and circuits are assumed to be vertex self-avoiding. The interior of a circuit is the bounded component of its complement, when the circuit is viewed as a Jordan curve in the plane, and the interior is a subset of $\mathbb{R}^2$. The dual graph of $\mathbb{Z}^2$ is written as $((\mathbb{Z}^2)^*, (\mathcal{E}^2)^*)$, where
\[
(\mathbb{Z}^2)^* = \mathbb{Z}^2 + \left( \frac{1}{2}, \frac{1}{2}\right) \text{ and } (\mathcal{E}^2)^* = \mathcal{E}^2 + \left( \frac{1}{2},\frac{1}{2}\right).
\]
For $x \in \mathbb{Z}^2$, the vertex dual to $x$, written $x^*$, is defined as $x+(1/2,1/2)$, and for $e \in \mathcal{E}^2$, the edge dual to $e$, written $e^*$, is the unique element of $(\mathcal{E}^2)^*$ which bisects $e$. The percolation model on the original lattice induces one on the dual lattice in the natural way: a dual edge $e^*$ is said to be $p$-open (for $p \in [0,1]$) if $e$ is, and is said to be $p$-closed otherwise.

One relation between invasion percolation and Bernoulli percolation is the following: if the invasion intersects a $p_c$-open cluster (maximal connected set of $p_c$-open edges), it must contain the whole cluster. Indeed, if it were to intersect such a cluster but not contain the entire cluster, then for all large $n$, there would be a $p_c$-open edge on the edge boundary of $G_n$. Due to the invasion's greedy algorithm, it therefore would only invade edges that are $p_c$-open for all large times, and this implies that there is an infinite $p_c$-open cluster, a contradiction. As a consequence of this fact, we obtain
\begin{equation}\label{eq: contain_circuits}
\text{a.s., }I\text{ contains all }p_c\text{-open circuits around the origin.}
\end{equation}

A central tool used to study invasion percolation is correlation length and we take its definition from \cite[Eq.~1.21]{kestenscaling}. For $n \in \N$ and $p \in (p_c , 1 ] $, let 
\[ 
\sigma(n,m,p) = \P( \text{there is a } p \text{-open left-right crossing of } [0,n] \times [0,m] ), 
\]
where the term ``$p$-open left-right crossing of $[0,n] \times [0,m]$'' means a path in $[0,n] \times [0,m]$ with all edges $p$-open which joins some vertex in $\{0\} \times [0,m]$ to some vertex in $\{n\} \times [0,m]$. For $ \epsilon > 0 $ and $p > p_c ,$ we define
\begin{equation*}
L(p,\epsilon) = \min \{ n : \sigma (  n , n , p) > 1 - \epsilon \} .
\end{equation*}
$L(p,\epsilon)$ is called the (finite-size scaling) correlation length.  It is known that $\lim_{p \downarrow p_c} L(p,\epsilon) = \infty$ for $\epsilon>0$ and that there exists $\epsilon_1 > 0$ such that for all $0 < \epsilon , \epsilon' \leq \epsilon_1 $, one has $L(p,\epsilon) \asymp L(p,\epsilon')$ as $p \downarrow p_c$. We will therefore define $L(p) = L(p,\epsilon_1) $ with this fixed $\epsilon_1$ for simplicity. For $n \geq 1 $, let
\begin{equation*}
p_n = \min \{ p > p_c : L(p) \leq n \},
\end{equation*}
and define
\[
q_k = p_{3^{k}} \text{ for } k \geq 0.
\]
We note here that 
\begin{equation}\label{eq: comparable_sum}
\sum_{k=1}^n F^{-1}(q_k) \asymp \sum_{k=1}^n F^{-1}\left( \frac{1}{2} + \frac{1}{2^k}\right) \text{ as } n \to \infty.
\end{equation}
This follows from the fact that $n^{-\delta_0} < p_n-p_c < n^{-\epsilon_0}$ for some $\delta_0,\epsilon_0>0$ and $n \geq 2$ (explained in \cite[Eq.~(2.5)]{Asymptotics}) and from monotonicity of $F^{-1}$ (for instance, see \cite[Lemma~4.1]{Asymptotics}).

We list the following properties of correlation length, with references to their proofs.
\begin{enumerate}
%\item There exists $K>0$ such that
%\begin{equation}
%K n \leq L(p_n) \leq n \text{ for all } n \geq 1.
%\end{equation}
%(See \cite{??}.)

%\item
%\[
%\P( \hat{p}_n > p ) \leq K_1 \exp \Big ( - K_2 \dfrac{3^n}{L(p)} \Big ) \text{ for all } n \geq 1 ,   p > p_c , 
%\]

\item \label{eq: comparable_p_n} There exist $c_{1.2.1},C_{1.2.1} > 0 $ such that
\begin{equation}\label{eq: comparable_p_n_equation}
c_{1.2.1} \Big | \log \dfrac{m}{n} \Big | \leq \Big | \log \dfrac{p_m - p_c}{p_n - p_c} \Big | \leq C_{1.2.1} \Big |  \log \dfrac{m}{n} \Big | \text{ for all } m,n \geq 1.
\end{equation}
This is a consequence of \cite[Prop.~34]{NolinNearCritical}.
\item There exists $c_{1.2.2}>0$ such that for all $n \geq 1$,
\begin{equation}\label{eq: L_p_n}
c_{1.2.2} n \leq L(p_n) \leq n.
\end{equation}
This follows from \cite[Eq.~(2.10)]{Jarai}.
\item There exist $c_{1.2.3}, C_{1.2.3}>0$ such that for all $p>p_c$,
\begin{equation}\label{eq: scaling_relation}
c_{1.2.3} \leq L(p)^2 \pi_4(L(p)) (p-p_c) \leq  C_{1.2.3},
\end{equation}
where $\pi_4(n)$ is the probability that there are two vertex-disjoint (except their initial points) $p_c$-open paths connecting the origin to $\partial B(n)$, and two vertex-disjoint (except for their initial points) $p_c$-closed dual paths connecting the point $(1/2,1/2)$ to $\partial B(n)$. This relation appears as \cite[Prop.~34]{kestenscaling}.
\item From \cite[Section~4]{nguyen}, there exists $c_{1.2.4}>0$ such that for all $n \geq 1$,
\begin{equation}\label{eq: connect_to_infinity}
\mathbb{P}(B(n) \text{ is connected to }\infty \text{ by a }p_n\text{-open path}) \geq c_{1.2.4}.
\end{equation}
Here, ``$\ldots$ connected to $\infty \ldots$'' means that there is an infinite vertex self-avoiding $p_n$-open path starting in $B(n)$.
\end{enumerate}

%\textbf{DEFINITIONS NEEDED TO DEFINE $E$}

\subsection{Outline of proof}\label{sec: outline}

The proof of Theorem~\ref{thm: main_thm} is split into two cases. At the end of Section~\ref{sec: end}, we assume that $\sum_{k}F^{-1}(q_k) < \infty$, and we explicitly construct an event $A$ (whose definition is below \eqref{eq: R_def}) with positive probability such that on $A$, for all $n \geq R$,
\[ 
T^{\text{inv}}(0, \partial B(n) ) - T(0, \partial B(n) ) \geq b. 
\]
Here $b,R$ are positive constants. This is sufficient to show that for $n \geq R$,
\[
\mathbb{E} (T^{\text{inv}}(0, \partial B(n) ) - T(0, \partial B(n) ) ) \geq b \mathbb{P}(A) > 0,
\]
and this is at least a constant times $\sum_k F^{-1}(q_k)$. The comparison \eqref{eq: comparable_sum} then finishes the proof in this case.

For the rest of the outline, we therefore assume that $\sum_{k}F^{-1}(p_{3^{k}}) = \infty$. For large $n$, we consider subannuli of $B(n)$ of the form $\Ann(3^k,3^{k+3} ) = B(3^{k+3}) \setminus B(3^k)$ for $k = 0,...,  \floor{ \log_3 n } - 3$ and in Section~\ref{sec: E_k_def} define events $(E_k)$, which are illustrated in Figure~\ref{fig:Config1}, depending on the state of edges in these annuli. Two of the paths involved in the definition of $E_k$ are a $p_c$-open circuit around the origin in $\Ann(3^{k},3^{k+1})$ and another $p_c$-open circuit around the origin in $\Ann(3^{k+2}, 3^{k+3})$ (see $\gamma_1^1$ and $\gamma_2^1$ in Definition~\ref{def: E_k}). Letting $C_k$ and $D_k$ be the outermost and innermost such circuits respectively, the fact that they have zero total weight and are contained in the invasion (see \eqref{eq: contain_circuits}) implies that the difference $\Delta = T^{\text{inv}} - T$ satisfies
\[
\Delta(0,\partial B(n)) \geq \Delta (C_k, D_k)\mathbf{1}_{E_k}
\]
and furthermore (see \eqref{eq: pizza_head})
\[
\Delta(0,\partial B(n)) \geq \Delta(C_1,D_1) \mathbf{1}_{E_1} + \Delta(C_4,D_4) \mathbf{1}_{E_4} + \cdots + \Delta(C_r,D_r)\mathbf{1}_{E_{3r+1}},
\]
where $r$ is the largest integer with $3^{3r+4} \leq n$. (Here we consider only $E_k$'s with values of $k$ differing by at least 3 to ensure that their associated annuli are disjoint.) 

%On the event $E_k$, let us denote $\mathcal{C}_k^1$ as the innermost $p_c$-open circuit surrounding $B(3^{3k})$ contained within $\Ann( 3^{3k},3^{3k+1} )$ and $\mathcal{C}_k^2$ as the innermost $p_c$-open circuit surrounding $B(3^{3k+2})$ contained within $\Ann( 3^{3k+2},3^{3k+3} )$. Let $\gamma$ be a geodesic from the origin to the boundary of $B(3^n)$ and notice that the segment of $\gamma$ between the circuits $\mathcal{C}_k^1$ and $\mathcal{C}_k^2$ is itself a geodesic from $\mathcal{C}_k^1$ to $\mathcal{C}_k^2$. Notice further that $T^{\text{inv}} \geq T $ so we are able to lower bound the difference between the invasion restricted passage time and the ordinary passage time, which we denote by
%\[ 
%\Delta(0,\partial B(3^n) ) = T^{\text{inv}}(0,\partial B(3^n) ) - T(0,\partial B(3^n) ),
%\]
%
%in terms of the contribution coming from the passage time from between the pairs of $p_c$-open circuits,
% \[ \Delta(0,\partial B(3^n) ) \geq \Delta(\mathcal{C}_1^1,\mathcal{C}_1^2) \textbf{1}_{E_1} + .... +\Delta(\mathcal{C}_{\floor{ (n-3)/3 } }^1,\mathcal{C}_{\floor{ (n-3)/3 } }^2) \textbf{1}_{E_{\floor{ (n-3)/3 } }} .\]
 
To bound the terms in the sum, we define a set of ``good'' indices $G = \{ k : F^{-1}(q_k) < 2 F^{-1}(q_{k+1})\}$ and we show in \eqref{eq: part_one} and \eqref{eq: part_two} that for such values of $k$, if $E_k$ occurs, then the passage time $T^{\text{inv}}(C_k,D_k)$ is at least $3 F^{-1}(q_{k+1})$, while the ordinary passage time $T(C_k,D_k)$ is at most $2F^{-1}(q_{k+1}) $. This is possible because on $E_k$, any path in the invasion that crosses $\text{Ann}(3^{k+1}, 3^{k+2})$ must contain the edges $e_1,e_2,e_3$ (which have weights $\geq q_{k+1}$) shown in Figure~\ref{fig:Config1}, whereas an unrestricted path may simply take edge $e_4$ (which has weight $\leq q_k$). This implies that
 \[ 
 \E \Delta(C_k,D_k) \textbf{1}_{E_k} \geq  F^{-1}(q_{k+1}) \mathbb{P}(E_k),
 \]
and combining this with the above inequality,
\[ 
\E \Delta(0,\partial B(n)) \geq  \left( \inf_\ell \mathbb{P}(E_\ell)\right) \times \sum_{ \substack{ k : 3k+1 \in G \\ 3k+3 \leq \lfloor \log_3 n \rfloor }}F^{-1}(q_{3k+4}).
\]
Similarly, we can obtain
\[ 
\E \Delta(0,\partial B(n)) \geq  \frac{1}{3} \left( \inf_\ell \mathbb{P}(E_\ell)\right) \times \sum_{ \substack{ k  \in G \\ k+3 \leq \lfloor \log_3 n \rfloor }}F^{-1}(q_{k+1}).
\]
(Compare to \eqref{eq: last_part}.) In Section~\ref{sec: E_k_bound}, we show that the infimum is positive, and so because the definition of $G$ entails that
\[ 
\sum_{ \substack{ k : k \in G \\ k+3 \leq \lfloor \log_3 n \rfloor }}F^{-1}(q_{k+1}) \asymp \sum_{ k: k+3 \leq \lfloor \log_3 n \rfloor }F^{-1}(q_k)
\] 
(from Lemma~\ref{lem: Good_Indices}), we can finish the proof with another application of \eqref{eq: comparable_sum}.

\section{Proof of Theorem~\ref{thm: main_thm}}\label{sec: main_proof}

\subsection{Step 1: Definition of $E_k$}\label{sec: E_k_def}

%\textbf{DEFINE E}

In this section, we define events $(E_k)_{k \geq 0}$ whose occurrence allows us to give a lower bound for $T^{\text{inv}}-T$. To state this bound precisely, we define $C_k$ to be the outermost $p_c$-open circuit around the origin in $\text{Ann}(3^k,3^{k+1})$ (if it exists) and let $D_k$ be the innermost $p_c$-open circuit around the origin in $\text{Ann}(3^{k+2},3^{k+3})$. (On $E_k$, these circuits will always exist --- see the first two bullet points of Definition~\ref{def: E_k}.)
%More precisely, for $n \geq 0$ and $k = 0, \dots, \lfloor \log_3 n\rfloor-3$, we set 
%\begin{equation}\label{eq: T_k_def}
%T_k^{\text{inv}}(n) = \sum_{e \in \gamma_n^{\text{inv}} \cap \mathcal{E}_k} t_e \text{ and } T_k(n) = \sum_{e \in \gamma_n \cap \mathcal{E}_k}t_e,
%\end{equation}
%where $\gamma_n^{\text{inv}}$ is a (deterministically chosen) geodesic for $T^{\text{inv}}(0, \partial B(n))$, $\gamma_n$ is a (deterministically chosen) geodesic for $T(0,\partial B(n))$, and $\mathcal{E}_k$ is the set of edges $E(B(3^{k+2})) \setminus E(B(3^{k+1}))$, where $E(B(m))$ is the set of edges with both endpoints in $B(m)$. 
The event $E_k$ will be constructed so that for $n \geq 0$ and $k=0, \dots, \lfloor \log_3 n \rfloor - 3$, if $k$ is in a certain ``good'' set of indices
\begin{equation}\label{eq: G_def}
G = \{k \geq 0 : F^{-1}(q_k) \leq 2F^{-1}(q_{k+1})\},
\end{equation}
then
\begin{equation}\label{eq: gain_on_E_k}
(T^{\text{inv}}(C_k,D_k)-T(C_k,D_k)) \mathbf{1}_{E_k} \geq F^{-1}(q_{k+1}) \mathbf{1}_{E_k} \text{ a.s.}
\end{equation}
(Recall that $C_k$ and $D_k$ are contained in the invasion by \eqref{eq: contain_circuits}.)

In the following definition, we use the notation $R(N) = [0,N] \times [0,N]$ for $N \geq 0$ and $\Ann(m,n) = B(n) \setminus B(m)$ for $0 \leq m \leq n$. Because there are many conditions comprising the event $E_k$, we encourage the reader to consult Figure~\ref{fig:Config1} for an illustration.

\begin{Def}\label{def: E_k} For $k \geq 0$ and real numbers $\alpha,\beta$ with $\alpha > 1$ and $\beta \in [0,1)$, we define the event $E_k = E_k(\alpha,\beta)$ that the following conditions hold.
\begin{itemize}
\item There is a $p_c$-open circuit, $\mathcal{\gamma}_1^1$,  in $\Ann(3^k,3^{k+1})$, which contains $B(3^k)$ in its interior.
\item There is a $p_c$-open circuit, $\mathcal{\gamma}_2^1$,  in $\Ann(3^{k+2},3^{k+3})$, which contains $B(3^{k+2})$ in its interior.
\end{itemize}
There are edges
\begin{itemize}
\item  $e_1 \in B_1 := \big ( - \frac{1}{2} 3^k , \frac{3}{2} 3^{k} \big ) + R(3^k) $ with $\omega_{e_1} \in (q_{k+1},p_c+\alpha(q_{k+1}-p_{c}) ) $,
\item  $e_2 \in B_2 := \big ( - \frac{1}{2} 3^k , \frac{5}{2} 3^{k} \big ) + R(3^k) $ with  $\omega_{e_2} \in (q_{k+1},p_c+\alpha(q_{k+1}-p_{c}) ) $,
\item  $e_3 \in B_3 := \big ( - \frac{1}{2} 3^k , \frac{7}{2} 3^{k} \big ) + R(3^k) $ with $\omega_{e_3} \in (q_{k+1},p_c+\alpha(q_{k+1}-p_{c}) ) $, and
\item  $e_4 \in B_4 := \big (  \frac{1}{2} 3^{k+1} , -\frac{1}{2} 3^{k+1} \big ) + R(3^{k+1}) $ with $\omega_{e_4} \in (p_c + \beta(q_{k}-p_{c}) ,q_{k}) $,
\end{itemize}
such that
\begin{itemize}
\item there is a $p_c$-open path which connects $\gamma_1^1$ to one endpoint of $e_1$, 
\item there is a $p_c$-open path which connects the other endpoint of $e_1$ to one endpoint of $e_2$, 
\item there is a $p_c$-open path which connects the other endpoint of $e_2$ to one endpoint of $e_3$, 
\item there is a $p_c$-open path which connects the other endpoint of $e_3$ to the $p_c$-open circuit $\gamma_{2}^1$,
\item there is a $p_{c}$-open path which connects $\gamma_1^1$ to $e_4$,
\item there is a $p_{c}$-open path which connects the other endpoint of $e_4$ to $\gamma_2^1 $, 
\item there is a $q_{k}$-closed dual path, $\gamma_{1}^{2}$, which connects one endpoint of $e_1^*$ to one endpoint of $e_3^*$, 
\item there is a $q_{k}$-closed dual path, $\gamma_{2}^{2}$, which connects the other endpoint of $e_3^*$ to the other endpoint of $e_1^*$,
\item there is a $q_{k}$-closed dual path, $\gamma_{1}^{3}$, which connects one endpoint of $e_4^*$ to one endpoint of $e_2^*$, 
\item there is a $q_{k}$-closed dual path, $\gamma_{2}^{3}$, which connects the other endpoint of $e_2^*$ to the other endpoint of $e_4^*$, and
\item there is a $q_{k+1}$-open path which connects $\gamma_{2}^{1}$ to $\infty$. 
\end{itemize}
Moreover,
\begin{itemize}
\item $\{e_4^*\} \cup \gamma_1^3 \cup \{e_2^*\} \cup \gamma_2^3$ forms a dual circuit around zero, and
\item $\{e_1^*\} \cup \gamma_2^2 \cup \{e_3^*\} \cup \gamma_1^2$ forms a dual circuit around $e_2^*$. 
\end{itemize}
\end{Def}

\begin{center}
	\begin{figure}
  \includegraphics[width=0.6\linewidth]{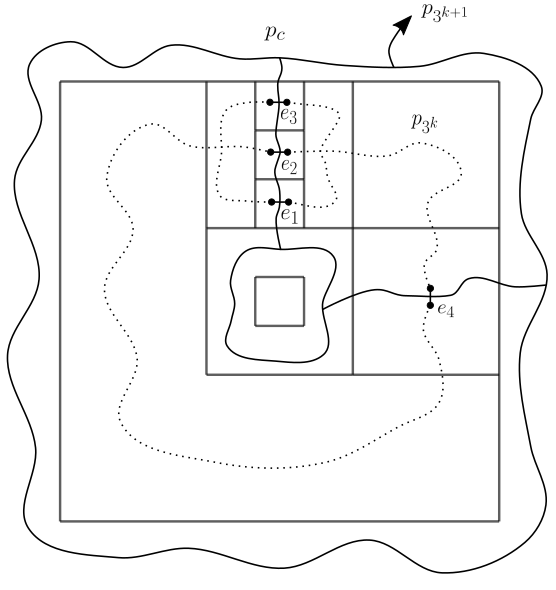}
  \caption{Illustration of the event $E_k$. The innermost box is $B(3^k)$ and the outermost is $B(3^{k+2})$. The solid lines represent $p_c$-open paths, the dotted lines represent $q_{k}$-closed dual paths, while the curve with the arrow indicates a $q_{k+1}$-open path to infinity. When $E_k$ occurs, any path connecting $B(3^{k+1})$ to $\partial B(3^{k+2})$ remaining in the invasion $I$ must contain the edges $e_1,e_2,e_3$ (not $e_4$), whereas a path in the original FPP model may take the edge $e_4$.}
  \label{fig:Config1}
\end{figure}
\end{center}

From this point forward, we pick $\alpha $ and $\beta $ satisfying the inequality in the next lemma.
\begin{Lemma}\label{lem: alpha_beta_choice}
There exist $\alpha,\beta$ with $\alpha>1>\beta>0$ such that for all $k \geq 0$,
\[
p_c+\alpha(q_{k+1}-p_{c}) < p_c + \beta (q_{k}-p_c). 
\] 
\end{Lemma}

\begin{proof}
Using \eqref{eq: comparable_p_n_equation}, we see that
\[ 
3^{c_{1.2.1}} < \dfrac{q_k - p_c}{q_{k+1} - p_c}.
\] 
So we choose $\alpha$ and $\beta$ sufficiently close to 1 that $\alpha < 3^{c_{1.2.1}} \beta$, and this implies
\[
p_c + \alpha(q_{k+1}-p_c) < p_c + \beta \cdot 3^{c_{1.2.1}}(q_{k+1}-p_c) < p_c + \beta (q_k-p_c).
\]
\end{proof} 

We will bound the probability $\mathbb{P}(E_k)$ from below in the next step. To finish the current step, we estimate the difference $T^{\text{inv}} - T$ when $E_k$ occurs; that is, we now prove inequality \eqref{eq: gain_on_E_k}. So suppose that $n \geq 0$ and $0 \leq k \leq \lfloor \log_3 n \rfloor - 3$. We will show that 
\begin{equation}\label{eq: part_one}
T^{\text{inv}}(C_k,D_k)\mathbf{1}_{E_k} \geq 3F^{-1}(q_{k+1})\mathbf{1}_{E_k} \text{ a.s.}
\end{equation}
and
\begin{equation}\label{eq: part_two}
T(C_k,D_k)\mathbf{1}_{E_k} \leq F^{-1}(q_k)\mathbf{1}_{E_k} \text{ a.s.}
\end{equation}
If we prove these two inequalities, then, under the additional assumption that $k \in G$, we would obtain
\[
(T^{\text{inv}}(C_k,D_k)-T(C_k,D_k))\mathbf{1}_{E_k} \geq (3F^{-1}(q_{k+1})-F^{-1}(q_k))\mathbf{1}_{E_k} \geq F^{-1}(q_{k+1})\mathbf{1}_{E_k},
\]
and this would show \eqref{eq: gain_on_E_k}.

We begin by proving \eqref{eq: part_one}, and to do this, we show that on the event $E_k$, any optimal path $\gamma_k^{\text{inv}}$ for $T^{\text{inv}}(C_k,D_k)$ must contain the edges $e_1,e_2,$ and $e_3$. Since these edges have weight $t_{e_i} \geq F^{-1}(q_{k+1})$, we would then obtain $T^{\text{inv}}(C_k,D_k) \geq t_{e_1}+t_{e_2}+t_{e_3} \geq 3F^{-1}(q_{k+1})$. The argument is similar for all three edges, so we show that $\gamma_k^{\text{inv}}$ contains $e_2$. Since $\gamma_k^{\text{inv}}$ crosses $\text{Ann}(3^{k+1},3^{k+2})$, by duality it must contain a edge $e$ whose dual is in $\{e_4^*\} \cup \gamma_1^3 \cup \{e_2^*\} \cup \gamma_2^3$. However, after the invasion touches the circuit $\gamma_1^1$ for the first time, it has access to infinitely many $p_c+\alpha(q_{k+1}-p_c)$-open edges (through the edges $e_1, e_2,e_3$). Because $C_k$ does not intersect the interior of $\gamma_1^1$ (it lies ``on or outside'' $\gamma_1^1$) all the edges of $\gamma_k^{\text{inv}}$ must then be $p_c+\alpha(q_{k+1}-p_c)$-open. Since $\gamma_1^3$ and $\gamma_2^3$ are $q_k$-closed, $e_4$ is $p_c+\beta(q_k-p_c)$-closed, $e_2$ is $p_c+\alpha(q_{k+1}-p_c)$-open, and Lemma~\ref{lem: alpha_beta_choice} implies that
\[
p_c + \alpha(q_{k+1}-p_c) < p_c + \beta(q_k-p_c) < q_k,
\]
it follows that $e_2 \in \gamma_k^{\text{inv}}$. This shows \eqref{eq: part_one}.

To complete this step, note that because $C_k$ is ``on or outside'' $\gamma_1^1$ and $D_k$ is ``on or inside'' $\gamma_2^1$, there is a path $\pi$ (through $e_4$)
%Let $\pi$ be the portion of $\gamma_k$  from its first intersection with $\gamma_1^1$ to its last intersection with $\gamma_8^1$. Then on the event $E_k$, we have $T_k(n) \leq T(\pi)$. Note that $\pi$ is an optimal path (geodesic) for $T(\gamma_1^1,\gamma_8^1)$ because it is a geodesic between its endpoints and all weights on $\gamma_1^1$ and $\gamma_8^1$ are zero. 
from $C_k$ to $D_k$ with passage time equal to $F^{-1}(q_k)$. This implies $T(C_k,D_k) \leq T(\pi) \leq F^{-1}(q_k)$, which is \eqref{eq: part_two}.

%{\color{blue} MOVED: After the invasion invades $\gamma_{1}^{1}$ it has a $p_c+\alpha(q_{k+1}-p_{c})$-open path to infinity. Thus, every edge the invasion takes there after must be $p_c+\alpha(q_{k+1}-p_{c})$-open. Any such connected path to infinity must contain the edges $e_1$,$e_2$, and $e_3$. Therefore the contribution to the cost of the invasion restricted geodesic in $\Ann(3^k,3^{k+2})$ is at least $3 F^{-1}(q_{k+1}) $. There is a path which costs $F^{-1}(q_k)$ and thus the ordinary geodesics costs no more than $F^{-1}(q_k)$. Therefore get the result. 
%}

\subsection{Step 2: Lower bound on $\mathbb{P}(E_k)$.}\label{sec: E_k_bound}
%%%%%%%

%\textbf{PROVE $P(E) > C $}
%\textbf{DEFINTE $A_e$, $\tilde{A}_e$ etc}

%\textbf{Proposition:} The events $(E_k)_k$ are bounded away from zero in probability. 
%In this section, we lower bound the probability of $E_k$:

\begin{Prop}\label{prop: EkBoundedFromZero} 
There exists $c_{ 2.2.1 } > 0$ so that for all $k \geq 0$, $\P(E_k) \geq c_{ 2.2.1 }$.
\end{Prop}

\begin{proof} 
To give a lower bound for the probability of $E_k$, we use several gluing constructions, themselves composed of the RSW theorem, the (generalized) FKG inequality, and Kesten's arms separation method. Because these arguments are now standard, we will confine ourselves to a rough outline of the approach. The interested reader should pay close attention to Figure~\ref{fig:Config1} throughout the sketch.

For $i=1, \dots, 4$, let $F_i$ be the event that in the box $B_i$ in the definition of $E_k$, there exists an appropriate four-arm edge in the central box of half the size of $B_i$. Specifically, defining $B_i'$ to be the box with half the sidelength of $B_i$ but with the same center, we let $F_i$ be the event that there is an edge $e_i \in B_i'$ such that $\omega_{e_i} \in I_i$, $e_i$ is connected to the top and bottom sides of $B_i$ by two vertex-disjoint $p_c$-open paths, and $e_i^*$ is connected to the left and right sides of $B_i$ by two vertex-disjoint $q_k$-closed paths, where $I_i = (q_{k+1}, p_c + \alpha(q_{k+1}-p_c))$ for $i=1, 2, 3.$ Define $F_4$ similarly, but with $I_4 = (p_c + \beta(q_k-p_c), q_k)$, the $p_c$-open paths touching the left and right sides of $B_4$, and the $q_k$-closed dual paths touching the top and bottom sides.

We claim that for some $c_{2.2.2}>0$, one has 
\begin{equation}\label{eq: claim_1}
\mathbb{P}(F_i) \geq c_{2.2.2} \text{ for } i=1, \dots, 4, \text{ and for all }k \geq 0.
\end{equation} 
So fix such $i$ and $k$ and first note that for $e \subset B_i'$, if $F_e$ is the event that the edge $e$ satisfies the conditions described in the definition of $F_i$, then for distinct $e,f \subset B_i'$, the events $F_e$ and $F_f$ are disjoint. Therefore $\mathbb{P}(F_i) = \sum_{e \subset B_i'} \mathbb{P}(F_e)$. Because $L(q_k) \leq 3^k$ (from \eqref{eq: L_p_n}), one can use \cite[Lemma~6.3]{DamronRelations} to prove that $\mathbb{P}(F_e) \geq c_{2.2.3}\mathbb{P}(F_e')$ for some $c_{2.2.3} > 0 $, where $F_e'$ is defined similarly to $F_e$, but the $q_k$-closed paths are instead $p_c$-closed. Last, Kesten's arm separation method (see \cite[Theorem~11]{NolinNearCritical}) implies that $\mathbb{P}(F_e') \geq c_{2.2.4}|I_i|\pi_4(2^k)$ for some $c_{2.2.4} > 0 $, where $|I_i|$ is the length of the interval $I_i$ and $\pi_4$ is defined below \eqref{eq: scaling_relation}. Putting together these pieces, we obtain
\[
\mathbb{P}(F_i) = \sum_{e \subset B_i'} \mathbb{P}(F_e) \geq c_{2.2.3}\sum_{e \subset B_i'} \mathbb{P}(F_e') \geq c_{2.2.5}|I_i| \pi_4(2^k) 2^{2k} \geq c_{2.2.6} (q_{k+1}-p_c) \pi_4(2^{k+1}) 2^{2k}.
\]
for some $c_{2.2.5},c_{2.2.6} > 0 $. Using the scaling relation stated above in \eqref{eq: scaling_relation}, the right side is bounded below by $c_{2.2.7}>0$. This demonstrates the claim in \eqref{eq: claim_1}.

Now that we have constructed the four-arm edges in the boxes $B_i$, we need to create the other macroscopic connections. By the RSW theorem \cite{Russo, SeymourWelsh}, the FKG inequality, and independence, one has
\[
\mathbb{P}(J) \geq c_{2.2.8} \text{ for all } k \geq 0,
\]
for some $c_{2.2.8} > 0$, where $J$ is the event that the following occur: 
\begin{enumerate}
\item There is a $p_c$-open circuit around the origin in $\text{Ann}(3^k,3^{k+1})$ which is connected by a $p_c$-open path in this annulus to the bottom side of $B_1$ and by another $p_c$-open path in this annulus to the left side of $B_4$,
\item there is a $p_c$-open circuit around the origin in $\text{Ann}(3^{k+2},3^{k+3})$ which is connected by a $p_c$-open path in this annulus to the top side of $B_3$ and by another $p_c$-open path in this annulus to the right side of $B_4$, and
\item there are $q_k$-closed dual paths in the following regions: (a) one connecting the left side of $B_2$ to the bottom side of $B_4$, and one connecting the left side of $B_1$ to the left side of $B_3$, all in the component of $\text{Ann}(3^{k+1},3^{k+2}) \setminus \cup_i B_i$ that contains the point $(-2 \cdot 3^{k+1},0)$, and (b) one connecting the right side of $B_2$ to the top side of $B_4$, and one connecting the right side of $B_1$ to the right side of $B_3$, all in the component of $\text{Ann}(3^{k+1},3^{k+2}) \setminus \cup_i B_i$ that contains the point $(2\cdot 3^{k+1}, 2\cdot 3^{k+1})$.
\end{enumerate}
The paths described in $J$ must be ``connected'' to the four-arm edges described in the events $F_i$, and this is done with the generalized FKG inequality (see \cite[Lem.~13]{NolinNearCritical}). Specifically, if $\hat{J}$ is the event that $\cap_i F_i \cap J$ occurs, but with the additional stipulations that: 
\begin{enumerate}
\item the first open connection in $\text{Ann}(3^k,3^{k+1})$ described in item 1 of the definition of $J$ is connected in this annulus to the ``lower'' $p_c$-open arm in $B_1$ and the second is connected to the ``left'' $p_c$-open arm in $B_4$,
\item the first open connection in $\text{Ann}(3^{k+2},3^{k+3})$ described in item 2 of the definition of $J$ is $p_c$-connected in this annulus to the ``upper'' $p_c$-open arm in $B_3$ and the second is $p_c$-connected to the ``right'' $p_c$-open arm in $B_4$,
\item the ``upper'' $p_c$-open arm in $B_1$ is $p_c$-connected to the ``lower'' $p_c$-open arm in $B_2$, and the ``upper'' $p_c$-open arm in $B_2$ is $p_c$-connected to the ``lower'' $p_c$-open arm in $B_3$,
\item the first $q_k$-closed dual path described in item 3(a) of the definition of $J$ is $q_k$-connected to the ``left'' $q_k$-closed arm in $B_2$ and the ``bottom'' $q_k$-closed arm in $B_4$, and the second is $q_k$-connected to the ``left'' $q_k$-closed arm in $B_1$ and the ``left'' $q_k$-closed arm in $B_3$, and
\item the first $q_k$-closed path described in item 3(b) of the definition of $J$ is $q_k$-connected to the ``right'' $q_k$-closed arm in $B_2$ and the ``top'' $q_k$-closed arm in $B_4$, and the second is $q_k$-connected to the ``right'' $q_k$-closed arm in $B_1$ and the ``right'' $q_k$-closed arm in $B_3$,
\end{enumerate}
then
\[
\mathbb{P}(\hat J) \geq c_{2.2.9} \text{ for all } k \geq 0.
\]

Finally, we must combine the event $\hat J$ with the connection to infinity. Letting $H$ be the event that there is a $q_{k+1}$-open path connecting $B(3^{k+2})$ to infinity, then by \eqref{eq: connect_to_infinity}, one has $\mathbb{P}(H) \geq c_{2.2.10}$ for all $k \geq 0$. To combine this with $\hat J$, we again use the generalized FKG inequality. It implies that
\[
\mathbb{P}(\hat J \cap H) \geq c_{2.2.11} \text{ for all }  k \geq 0.
\]
Because $\hat J \cap H$ implies the event $E_k$, this completes the sketch of the proposition.
\end{proof}

\subsection{Step 3: Good indices and the end of the proof.}\label{sec: end}

In this last step of the proof, we first prove a lemma which will imply that in the case that $(x_k) = (F^{-1}(q_k))$ is not summable, the sum of $F^{-1}(q_k)$ over all $k \in G$ for $k \leq n$ is comparable to the sum over all $k \leq n$. Recall that $G$ is the ``good'' set of indices defined in \eqref{eq: G_def} for which the lower bound for $T^{\text{inv}}(C_k,D_k) - T(C_k,D_k)$ from \eqref{eq: gain_on_E_k} holds. We will use this lemma along with \eqref{eq: gain_on_E_k} to prove Theorem~\ref{thm: main_thm} afterward.

%% Good Indicies Lemma (Begin)
\begin{Lemma}\label{lem: Good_Indices}
 Let $(x_k)_{k \geq 0 }$ be a nonnegative monotone nonincreasing sequence.
%such that $\sum_{ k \geq 0 } x_k = + \infty $. 
Then for all $n \geq 0$,
%there exists $C_{\ref{lem: Good_Indicies} .1} > 0 $ and $N \geq 0$ so that for all $n \geq N $ we have
%\[ \sum_{k \leq n } x_k \leq C_{\ref{lem: Good_Indicies} .1} \sum_{ \substack{ k \in G \\ k \leq n } } x_k \] }
\[
\sum_{k \leq n} x_k \leq 3x_0 + 3 \sum_{\substack{k \in G \\ k \leq n}} x_{k+1},
\]
where $ G = \{ k : x_{k}/x_{k+1} < 2 \} $.
\end{Lemma}
\begin{proof}
If $k, k+ 1, \dots, k+m \in G^c$ and $1 \leq \ell \leq m+1$, one has
\[
x_{k+\ell} \leq \frac{x_{k+\ell-1}}{2} \leq \dots \leq \frac{x_k}{2^\ell},
\]
and so
\[
\sum_{\ell=0}^m x_{k+\ell} \leq x_k + x_k \sum_{\ell=1}^{m+1} 2^{-\ell} \leq 2 x_k.
\]
By partitioning $G^c$ into a collection of maximal disjoint intervals and applying this inequality to each such interval, we obtain
\[
\sum_{\substack{k \in G^c \\ k \leq n}} x_{k+1} \leq \sum_{\substack{k \in G^c \\ k \leq n}} x_k \leq 2x_0 + 2 \sum_{\substack{k \in G \\ k \leq n}} x_{k+1}.
\]
Here, we have used that if $k$ is the first element of an interval in $G^c$, then $x_{k-1} \in G$, unless $k=0$. Adding $x_0 + \sum_{k \in G, k \leq n} x_{k+1}$ to both sides completes the proof of the lemma.
\end{proof}
%%%% Good Indices (END)

%\textbf{FINISH IT OFF}

\begin{proof}[Proof of Theorem~\ref{thm: main_thm}]
Now we prove the main theorem. The main step is to show that for $i=0,1,2$,
\begin{equation}\label{eq: last_part}
T^{\text{inv}}(0,\partial B(n)) - T(0,\partial B(n))  \geq \sum_{ \substack{ k : 3k+i \in G \\ 3k+i \leq \lfloor \log_3 n\rfloor - 3 } } F^{-1}(q_{3k+i+1}) \mathbf{1}_{E_{3k+i}} .
\end{equation}
To justify this inequality, recall that for a given $k$ such that $E_k$ occurs, $C_k$ and $D_k$ have zero weight and are therefore in the invasion by \eqref{eq: contain_circuits}, so one has
\[
T(A,B) = T(A,C_k) + T(C_k,D_k) + T(D_k,B)
\]
and
\[
T^{\text{inv}}(A,B) = T^{\text{inv}}(A,C_k) + T^{\text{inv}}(C_k,D_k) + T^{\text{inv}}(D_k,B)
\]
for any $A \subset B(3^k)$ and $B \subset B(3^{k+3})^c$. For a given $i=0,1,2$, therefore, if $3k+i \leq \lfloor \log_3 n \rfloor - 3$, then
\begin{align*}
&(T^{\text{inv}}(0,\partial B(n)) - T(0,\partial B(n)))\mathbf{1}_{E_{3k+i}} \\
=~& \bigg[ (T^{\text{inv}}(0,C_{3k+i}) - T(0,C_{3k+i})) + (T^{\text{inv}}(C_{3k+i}, D_{3k+i}) - T(C_{3k+i}, D_{3k+i})) \\
+~& (T^{\text{inv}}(D_{3k+i}, \partial B(n)) - T(D_{3k+i}, \partial B(n))) \bigg] \mathbf{1}_{E_{3k+i}}.
\end{align*}
In this way, we decouple the passage times between circuits. Applying this idea to all the circuits $C_{3k+i}, D_{3k+i}$ for $3k+i \leq \lfloor \log_3 n \rfloor - 3$ with $3k+i \in G$, and using that $T^{\text{inv}} \geq T$, we obtain
\begin{equation}\label{eq: pizza_head}
T^{\text{inv}}(0,\partial B(n)) - T(0,\partial B(n)) \geq \sum_{\substack{k : 3k+i \in G \\ 3k+i \leq \lfloor \log_3 n \rfloor - 3}} (T^{\text{inv}}(C_{3k+i},D_{3k+i}) - T(C_{3k+i},D_{3k+i})) \mathbf{1}_{E_{3k+i}}.
\end{equation}
(Here we have chosen indices of the form $3k+i$ to ensure that the annuli associated to the events $E_{3k+i}$ are disjoint.) 
%Recall that $\gamma_n^{\text{inv}}$ is a deterministically chosen geodesic for $T^{\text{inv}}(0, \partial B(n))$ and $\gamma_n$ is a deterministically chosen geodesic for $T(0,\partial B(n))$, and also recall the definitions of $T_k^{\text{inv}}(n)$ and $T_k(n)$ in \eqref{eq: T_k_def}. Since the portion of $\gamma_n^{\text{inv}}$ (or $\gamma_n$) from its first intersection with $C_k$ to its last intersection with $D_k$ is a geodesic from $C_k$ to $D_k$ for $T^{\text{inv}}$ (or $T$),
%\[
%(T^{\text{inv}}(C_{3k+i},D_{3k+i}) - T(C_{3k+i},D_{3k+i})) \mathbf{1}_{E_{3k+i}} \geq (T_{3k+i}^{\text{inv}}(n) - T_{3k+i}(n)) \mathbf{1}_{E_{3k+i}}.
%\]
Combining this with \eqref{eq: gain_on_E_k}, we obtain \eqref{eq: last_part}.

Averaging \eqref{eq: last_part} over $i=0,1,2$ produces
\[
T^{\text{inv}}(0,\partial B(n)) - T(0,\partial B(n))  \geq \frac{1}{3} \sum_{ \substack{ k \in G \\ k \leq \lfloor \log_3 n\rfloor - 3 } } F^{-1}(q_{k+1}) \mathbf{1}_{E_k}.
\]
By Proposition~\ref{prop: EkBoundedFromZero}, this becomes
\[
\mathbb{E}(T^{\text{inv}}(0,\partial B(n)) - T(0,\partial B(n)))  \geq \frac{c_{2.2.1}}{3} \sum_{\substack{ k \in G \\ k \leq \lfloor \log_3 n\rfloor - 3 }} F^{-1}(q_{k+1}).
\]
Lemma~\ref{lem: Good_Indices} then implies
\[
\mathbb{E}(T^{\text{inv}}(0,\partial B(n)) - T(0,\partial B(n)))  \geq  \frac{c_{2.2.1}}{3}\left[  - F^{-1}(q_0) + \frac{1}{3} \sum_{ k \leq \lfloor \log_3 n\rfloor - 3 } F^{-1}(q_k) \right].
\]
Using \eqref{eq: comparable_sum}, this implies the inequality of Theorem~\ref{thm: main_thm} if $\sum_k F^{-1}(q_k) = \infty$.

\begin{center}
	\begin{figure}
  \includegraphics[width=0.6\linewidth]{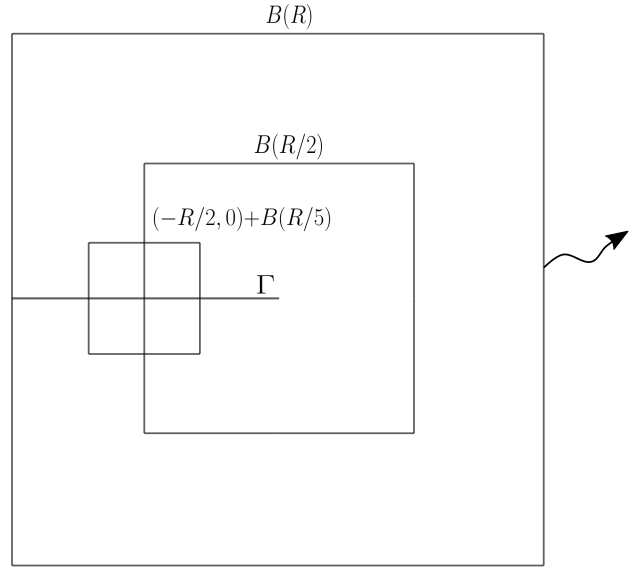}
  \caption{Illustration of the event $A$. The arrowed curve emanating from $B(R)$ has edges with weight $\leq a$ and all edges on $\partial B(R)$ have weight zero. The edges in the segment $\Gamma$ (which starts at the origin) have weight in the interval $[a/2,a]$, and edges touching $\partial B(R/2)$ (and in the box $(-R/2,0) + B(R/5)$) but not in $\Gamma$ have weight $\geq b$. All other edges in $B(R)$ have weight zero. On this event, any path from the origin to $\partial B(R)$ in the invasion must contain the segment of $\Gamma$ passing through $\partial B(R/2)$ on the left, and must therefore pick up at least weight $aR/5$.}
  \label{fig: fig_2}
\end{figure}
\end{center}

If $\sum_k F^{-1}(q_k) < \infty$, then we explicitly construct an event $A$ on which geodesics in $I$ have higher weight than true geodesics.  First pick $a,b$ with $0 < a < b$ such that $\mathbb{P}(t_e \in [a/2,a]) > 0$ and $\mathbb{P}(t_e \geq b) > 0$. (If this is impossible, then $\sum_k F^{-1}(q_k) =\infty$.) Fix an integer $R$ which is a multiple of 10 and satisfies
\begin{equation}\label{eq: R_def}
R \geq 10b/a,
\end{equation} 
and let $\Gamma$ be the set of edges of the form $\{(-n-1,0),(-n,0)\}$ for $0 \leq n \leq R-1$. Last, define $A$ to be the event that 
\begin{enumerate}
\item $B(R)$ is connected to infinity by a path of edges $e$ with $t_e \leq a$,
\item all edges $e$ with both endpoints in $\partial B(R)$ have $t_e=0$,
\item all edges $e \in \Gamma$ have $t_e \in [a/2,a]$,
\item all edges $e \notin \Gamma$ with one endpoint in $B(R/2)$ and one endpoint in $B(R/2)^c$ have $t_e \geq b$
\item all edges $e \notin \Gamma$ with both endpoints within $\ell^\infty$ distance $R/5$ of $(-R/2,0)$ have $t_e \geq b$, and
\item all other edges $e$ with both endpoints in $B(R)$ have $t_e=0$.
\end{enumerate}
(See Figure~\ref{fig: fig_2} for an illustration of the event $A$.) We claim that
\begin{equation}\label{eq: final_claim}
\mathbb{E}[T^{\text{inv}}(0, \partial B(n)) - T(0, \partial B(n))] \geq b\mathbb{P}(A) > 0 \text{ for } n \geq R.
\end{equation}
Assuming this claim, the statement of Theorem~\ref{thm: main_thm} follows from \eqref{eq: comparable_sum} if $\sum_k F^{-1}(q_k) < \infty$.

To show \eqref{eq: final_claim}, we show that the difference of passage times is at least $b$ on the event $A$. Because $A$ has positive probability (conditions (2)-(6) are clear, and for condition (1), we use that $\mathbb{P}(t_e \leq a) > 1/2$, and so with positive probability, any given vertex on $\partial B(R)$ is connected to infinity by a path outside $B(R)$ all whose edges have weight $\leq a$), this will complete the proof. First note that due to item (2), on $A$ we have
\[
T^{\text{inv}}(0, \partial B(n)) - T(0, \partial B(n)) \geq T^{\text{inv}}(0, \partial B(R)) - T(0, \partial B(R)) \text{ for } n \geq R.
\]
Next, since there is an infinite edge-self avoiding path starting at 0 whose edges have weight $\leq a$ (just follow $\Gamma$ to $\partial B(R)$ and then to infinity using item (1)), all edges $e$ in $I$ satisfy $t_e \leq a$. Therefore each path in $I$ connecting 0 to $\partial B(R)$ must contain all edges in $\Gamma$ with both endpoints within $\ell^\infty$ distance $R/5$ of $(-R/2,0)$. This implies by item (3) that on $A$, one has
\[
T^{\text{inv}}(0, \partial B(R)) \geq \frac{2R}{5} \cdot \frac{a}{2} = \frac{aR}{5}.
\]
On the other hand, there exists a path from 0 to $\partial B(R)$ with passage time equal to $b$: simply follow the positive $e_1$-axis. Therefore on $A$, one has
\[
T(0, \partial B(R)) \leq b.
\]
By the definition of $R$ in \eqref{eq: R_def}, $aR/5 - b \geq b$, and this shows \eqref{eq: final_claim}, completing the proof of Theorem~\ref{thm: main_thm}.
\end{proof}

\bigskip
\noindent
{\bf Acknowledgements.} The research of M. D. is supported by an NSF CAREER grant.


\begin{thebibliography}{9}

\bibitem{ADH}
Auffinger, A.; Damron, M.; Hanson. J. 50 years of first-passage percolation, {\it University Lecture Series, vol. 68,} American Mathematical Society, Providence, RI, 2017.

\bibitem{Asymptotics}
Damron, M.; Lam, W-L.; Wang, X. Asymptotics for $2$D critical first passage percolation. {\it Ann. Probab.} {\bf 45} (2017), 2941--2970.

\bibitem{CCD}
Chayes, J. T.; Chayes, L; Durrett, R. Critical behavior of the two-dimensional first-passage time. {\it J. Stat. Phys.} {\bf 45} (1986), 933--951.

\bibitem{DamronOutlet}
  Damron, M.; Sapozhnikov, A. Outlets of $2$D invasion percolation and multiple-armed incipient infinite clusters. {\it Probab. Theory Relat. Fields} \textbf{150} (2011), 257--294.
 
\bibitem{DamronInvasion}
  Damron, M.; Sapozhnikov, A. Limit theorems for 2D invasion percolation. {\it Ann. Probab.} \textbf{40} (2012), 893--920.
  
  \bibitem{DamronRelations}
  Damron, M.; Sapozhnikov, A.; V\'agv\"ogyi, B.
  Relations between invasion percolation and critical percolation in two dimensions. {\it Ann. Probab.} {\bf 37} (2009), 2297--2331.
  
  \bibitem{HW}
  Hammersley, J. M.; Welsh, D. J. A.
  First-passage percolation, subadditive processes, stochastic networks, and generalized renewal theory. {\it 1965 Proc. Internat. Res. Semin., Statist. Lab., Univ. California, Berkeley, Calif.} pp. 61--110 {\it Springer-Verlag, New York}.
  
  \bibitem{Jarai}
  J\'arai, A. A. Invasion percolation and the incipient infinite cluster in $2D$. {\it Commun. Math. Phys.} {\bf 236} (2003), 311--334.
  
  \bibitem{aspects}
  Kesten, H. Aspects of first-passage percolation. {\it \'Ecole d'\'et\'e de Probabilit\'es de Saint Flour XIV}, Lecture Notes in Mathematics, {\bf 1180} (1986), 125-264.
  
  \bibitem{kestenscaling}
  Kesten, H. Scaling relations for $2D$-percolation. {\it Commun. Math Phys.} {\bf 109} (1987), 109--156.
  
  \bibitem{nguyen}
  Nguyen, B. G. Correlation lengths for percolation processes. Ph. D. dissertation, Univ. California, 1985.
  
  \bibitem{NolinNearCritical}
  Nolin P. Near-critical percolation in two dimensions. {\it Electron. J. Probab.} {\bf 13} (2008), 1562--1623.
 
 \bibitem{Russo}
 Russo, L. A note on percolation. {\it Z. Warsch. und Verw. Gebiete.} {\bf 43} (1978), 39--48.
 
 \bibitem{SeymourWelsh}
 Seymour, P. D.; Welsh, D. J. A. Percolation probabilities on the square lattice. Advances in graph theory (Cambridge Combinatorial Conf., Trinity College, Cambridge, 1977). {\it Ann. Discrete Math.} {\bf 3} (1978), 277--245.
 
 \bibitem{Zhangdouble}
 Zhang, Y. Double behavior of critical first-passage percolation. In {\it Perplexing Problems in Probability. Progress in Probability}, {\bf 44} (1999), 143--158. Birkh\"auser, Boston, MA.
 
\end{thebibliography}
\end{document}